\documentclass[a4paper,12pt,twoside]{article}

\usepackage{a4 wide}

\usepackage[T1, T2A]{fontenc}
\usepackage[utf8]{inputenc}
\usepackage[english]{babel}

\usepackage{amsmath}
\usepackage{amssymb}
\usepackage{amsthm}
\usepackage{amscd}
\usepackage{cite}
\usepackage[matrix,arrow,curve]{xy}

\newtheorem{theorem}{Theorem}[section]
\newtheorem{lemma}[theorem]{Lemma}
\newtheorem{proposition}[theorem]{Proposition}
\newtheorem{corollary}[theorem]{Corollary}
\newtheorem{remark}[theorem]{Remark}
\newtheorem{definition}[theorem]{Definition}

\newcommand{\Sob}{\!\!\vbox{\hbox{\,\,\,\,\tiny $\circ$}\vspace*{-1.0ex} \hbox{ $H^{1}$}\vspace*{-0.0ex}}}

\newcommand{\SobSl}{\!\!\vbox{\hbox{\,\,\,\tiny $\circ$}\vspace*{-1.6ex} \hbox{  \scriptsize $H^{1}$}\vspace*{-0.0ex}}}

\title{Spectral boundary value problems for Laplace--Beltrami operator: moduli of continuity of eigenvalues under domain deformation}
\author{A.M. Stepin, I.V. Tsylin}

\begin{document}
\maketitle
\abstract{The paper is pertaining to the spectral theory of operators and boundary value problems for differential equations on manifolds. Eigenvalues of such problems are studied as functionals on the space of domains. Resolvent continuity of the corresponding operators is established under domain deformation and estimates of continuity moduli of their eigenvalues $\slash$ eigenfunctions are obtained provided the boundary of nonperturbed domain is locally represented as a graph of some continuous function and domain deformation is measured with respect to the Hausdorff--Pompeiu metric.}

\section{Introduction.}

Let $M$ be a smooth connected compact orientable Riemannian manifold (possibly with boundary), $\mathcal{A}$ --- elliptic differential operator of second order on $M$. For an open subset $\Omega \subsetneq M$, $\bar{\Omega} \cap \partial M = \varnothing$, the following eigenvalue problem
\begin{equation}
\label{SpecDir}
\mathcal{A} u = \lambda u, \,\,\,\,  u \in \!\!\vbox{\hbox{\,\,\,\,\tiny $\circ$}\vspace*{-1.0ex} \hbox{ $H^{1}$}\vspace*{-0.0ex}}(\Omega),
\end{equation}
is considered; here solutions are understood in the weak (variational) sense.

Our aim is to estimate moduli of continuity for eigenvalues $\{\lambda_k\}$ of (\ref{SpecDir}) with multiplicities taken into account and considered as functions of rough domain $\Omega$; among surveys on this and related topics we notice \cite{Bucur, HenrotEng, Lemenant}. 

The problem can be considered in the context of spectral stability of inverse operator $\mathcal{A}^{-1}$ under small perturbations of $\Omega$. In the framework of this approach it is sufficient to obtain estimate of convergence rate for corresponding inverse operators with respect to uniform topology (resolvent convergence \cite{Maslov}).

A convenient instrument for dealing with the resolvent convergence is (equivalent to it) convergence in the sense of Mosco (see \cite{Mosco, Tsylin}) of the spaces $\Sob(\Omega)$. In the case when $M \subset \mathbb{R}^d$ is compact, Frehse\cite{Frehse} obtained capacitive conditions equivalent to Mosco's convergence.

In nineties conditions imposed on variations of plane domains (w.r.t. Hausdorff metric) were investigated so that to ensure the uniform resolvent convergence (see \cite{Bucur}). However the obtained results claimed only the fact of convergence without quantitative estimates.

Beginning with 2002 some progress was achieved in proving such estimates in case of plane domains, mainly, due to Burenkov and Davies' approach \cite{Burenkov}, that was further developed in a series of papers by Burenkov, Lamberti, Lanza de Cristoforis (see a survey in \cite{BurenkovLamberti}). Method of these authors made it possible to estimate the upper semicontinuity in $\Omega$ for eigenvalues of the problem (\ref{SpecDir}) if variation of domains is restricted to some technical class.

In the present paper resolvent continuity of the boundary value problem ($\mathcal{A}u = f$, $u \in \Sob(\Omega)$) with respect to domain perturbation (see th. \ref{th_42}) is established and using this fact estimates for moduli of continuity are obtained for eigenvalues and eigenfunctions\footnote{in the sense of generalized angle (see, for instance, \cite{Kato}, IV.2)} of the problem (\ref{SpecDir}) provided that 1) the boundary of nonperturbed domain can be locally represented as a graph of some continuous function and  2) the perturbation of domains is measured by Hausdorff--Pompeiu metric $d^{\mathcal{HP}}$ (see sections \ref{sect22}, \ref{Haus_p}).

To formulate our result about eigenvalues we write $\partial \Omega \in C^{0,\omega}$ if there exists an atlas for the manifold $M$ such that the intersection of $\partial \Omega$ with each of its charts either is empty or can be represented in local coordinates as the graph of a function with modulus of continuity not exceeding $C\cdot \omega$, where $\omega$ is a nondecreasing semi-additive function such that $\omega(0) = 0$ and $C$ is a positive constant.

\begin{theorem}
\label{result}
Let $\mathcal{A}$ be a strongly elliptic operator with Lipschitz coefficients on the Riemannian manifold $(M,g)$, $\Omega_1$ be a domain in $M$, $\bar{\Omega}_1 \cap \partial M = \varnothing$, $\partial \Omega_1 \in C^{0,\omega}$, where $\omega$ is some modulus of continuity. Then there exist positive constants $C_n = C_n( \Omega_1, \mathcal{A}, M)$, $\delta_0  = \delta_0(\Omega_1, \omega, M)$ such that for any domain $\Omega_2 \subsetneq M$, satisfying condition $\epsilon = d^{\mathcal{HP}}(\Omega_1, \Omega_2)\leq \delta_0$ the following estimate holds:
$$
|\lambda_n(\Omega_1) - \lambda_n(\Omega_2)| \leq C_n \cdot ( \omega (\epsilon) + \epsilon).
$$
\end{theorem}

Besides, a generalization of Burenkov--Lamberti theorem \cite{BurLam} to the case of domains on manifold (cor. \ref{cor_bur}) is obtained. These results were announced in  \cite{StepinTsylin}.

\section{Basic notions.}

Everywhere below we assume that the Riemannian manifold $(M,g)$ is $C^{1,1}$--smooth connected orientable compact (possibly with boundary); coordinate homeomorphisms map in $\mathbb{R}^d$ ($d = \mathrm{dim} M$) endowed with the standard Euclidean norm $|\cdot|$:
$$
|x|^2 = \sum_{i} (x^i)^2, \,\,\,\,\,\, x = (x^1, \ldots, x^d) \in \mathbb{R}^d, \,\,\,\,\, e_d = (0, \ldots, 0, 1).
$$

\subsection{Conditions imposed on operator $\mathcal{A}$.}

Let $\mathcal{A}'$ be a differential operation on $(M,g)$ locally representable in the form
$$
- \frac{1}{\sqrt{\det g}} \partial_{i} \left( a^{ij} \sqrt{\det g} \, \partial_{j} u \right).
$$
Assume that coefficients $a^{ij}$ define a continuous symmetric section $\textbf{A}$ of $T^2 M$. Denote $\textbf{G}$ the section of $T^2 M$ associated with the Riemannian structure $g$; let bilinear forms $\textbf{A}_x$ and $\textbf{G}_x$ in $T^*_x M \times T^*_x M$ be values of the sections $\textbf{A}$ and $\textbf{G}$. We assume that the following conditions are fulfilled.

\begin{enumerate}
\item[\textbf{A1}] There is a positive constant $\alpha$ such that $\forall x \in M\,\, \forall \xi \in T^*_x M \Rightarrow \alpha \textbf{G}_x(\xi, \xi) \leq \textbf{A}_x (\xi, \xi);$
\item[\textbf{A2}] A sections $\textbf{A}$ belongs to the space $C^{0,1}(M)$; a norm in $C^{0,1}(M)$ is defined by fixing some finite subatlas $\{(U, \kappa_U)\}_{U \in \mathcal{U}}$
\begin{align*}
\| \textbf{A} \|^\mathcal{U}_{C^{0,1}(M)} \stackrel{\mathrm{def}}{=} & \max_{x \in M} \max_{\xi \in T^*_x M, \xi \neq 0} \frac{\textbf{A}_x(\xi,\xi)}{\textbf{G}_x(\xi,\xi)} + \sum_{U \in \mathcal{U}} \max_{ij} \left[ a^{ij}_U \right]_{C^{0,1}(\kappa_U(U))},\\
[v]_{C^{0,1}(\kappa_U(U))} = &\sup_{x,y \in \kappa_U(U) , x\neq y} \frac{|v(x) - v(y)|}{|x-y|}, \,\,\,\, v:\kappa_U(U) \to \mathbb{R},
\end{align*}
where $a^{ij}_U$ are the coordinates of sections $\textbf{A}$ with respect to mapping $\kappa_U$.
\end{enumerate}
All the norms $\| \cdot \|^\mathcal{U}_{C^{0,1}(M)}$ are equivalent as regards their dependence on the choice of finite subatlas.

As the scalar products in $\Sob(\Omega)$ and $L_2(\Omega)$, $\Omega \subsetneq M$, we choose:
$$
(u,v)_{\SobSl(\Omega)} = \int_{\Omega} \textbf{G}(\nabla u, \nabla v) d\mu, \,\,\,\, (u,v)_{L_2(\Omega)} = \int_\Omega uv d\mu,
$$
where measure $\mu$ is associated with the Riemann metric $g$. Let $\Phi$ be continuous in $\Sob(M)$ bilinear form defined by the differential operation $\mathcal{A}'$,
$$
\Phi(u,v) = \int_M \textbf{A}(\nabla u, \nabla v) \, d\mu.
$$
Since the form $\Phi$ is positive and bounded in $\Sob(\Omega)$, operator $\mathcal{A}$ associated with $\Phi$ is uniquely defined (see, for example, \cite{Kato}); a function $u\in \Sob(\Omega)$ is the weak solution of the boundary value problem
\begin{equation}
\label{ProbDir}
\mathcal{A} u = f, \,\,\,\, f \in L_2(\Omega)
\end{equation}
if and only if $\forall v \in \Sob(\Omega)$
\begin{equation}
\label{weak_p}
\Phi(u,v)= \int_M f v d\mu.
\end{equation}

\subsection{Domains with boundaries of the class $C^{0,\omega}$. \\ The $\omega$--cusp condition.}
\label{sect22}

For a mapping $\omega: \mathbb{R}_+ \to \mathbb{R}_+$ such that $\omega(r) - \omega(0)$ is nonnegative and semi-additive we set $\psi(r) = \sqrt{r^2 + \omega(r)^2}$, $\phi(r) = r + \omega(r)$. Let $B_{\rho}(x)\subset \mathbb{R}^d$ be the ball with center $x$ and radius $\rho$ in the norm $|\cdot|$; we will often use the notion $B_r(X) \stackrel{\mathrm{def}}{=} \cup_{x \in X} B_r(x)$ for a set $X \subset \mathbb{R}^d$.
\begin{definition}
\label{W_def1}
We say that an open set $\Omega \subset \mathbb{R}^d$ satisfies the uniform $\omega$--cusp condition with parameter $r$ at a point $x$ if there exists $\xi_x\in\mathbb{R}^d$, $|\xi_x| = 1$, such that
\begin{enumerate}
\item[\textbf{W1}] $\left[ \left(B_{3\psi(r)}(x) \cap \Omega\right) - \mathcal{C}_{\omega, r}(\xi_x)\right] \cap B_{2\psi(r)}(x) \subset \Omega$;
\end{enumerate}
where $\mathcal{C}_{\omega, r}(\xi_x)$ is obtained from $\mathcal{C}_{\omega, r}(e_d)\stackrel{\mathrm{def}}{=}\mathcal{S}_{\omega, r}(e_d) \cup \mathcal{F}_{\omega, r}(e_d)$ by a rotation superposing $e_d$ on $\xi_x$ where $\mathcal{F}_{\omega, r}(e_d) = \left\{ z = (\tilde{z}, z^d) \in \mathbb{R}^d:\,\,  |z| < \psi(r),\,\,\, z^d \geq \omega(r) \right\}$, $\mathcal{S}_{\omega, r}(e_d) = \left\{(\tilde{z}, z^d) \in\mathbb{R}^d : \omega(|\tilde{z}|) < z^d < \omega(r),\,\,\, |\tilde{z}| < r \right\}$, $\tilde{z}\in\mathbb{R}^{d-1}$.
\end{definition}

Note that the condition \textit{\textbf{(W1)}} is equivalent to the following condition
\begin{enumerate}
\item[\textbf{W2}] $\left[\left(B_{3\psi(r)}(x) \backslash \Omega\right) + \mathcal{C}_{\omega, r}(\xi_x)\right] \cap \left( B_{2\psi(r)}(x) \cap \Omega\right) = \varnothing$.
\end{enumerate}
In fact, by virtue of symmetric of \textit{\textbf{(W1)}} and \textit{\textbf{(W2)}} with respect to changing $\Omega$ by its complement it is sufficient to check the implication \textit{\textbf{(W2)}} $\Rightarrow$ \textit{\textbf{(W1)}}. Otherwise
$$\exists y \in \left(B_{3\psi(r)}(x) \cap \Omega \right): B_{2\psi(r)} \cap \left( y - \mathcal{C}_{\omega, r}(\xi_x) \right) \not\subset \Omega$$
and hence there exists a point $z \in \partial \Omega \cap \left( y - \mathcal{C}_{\omega, r}(\xi_x) \right)$. Equivalently this means that $y \in z + \mathcal{C}_{\omega, r}(\xi_x)$; by \textit{\textbf{(W2)}} this inclusion leads to a contradiction: $y \in M \backslash \Omega$.
\begin{remark}
In the definition above we do not assume that $\omega(0) = 0$. All the propositions below are valid without this assumption if it is not stated explicitly.
\end{remark}

For a matrix $A \in C^{0,1}(\bar{\Omega})$, $\Omega \subset \mathbb{R}^d$ define the norm:
$$
\| A \|_{C^{0,1}(\bar{\Omega})} \stackrel{\mathrm{def}}{=} \left\| \left| A \right|_2\right\|_{L_\infty(\Omega)} +  \left\| \max_j \left| \partial_j A \right|_2 \right\|_{L_{\infty}(\Omega)}, \,\, \left| A \right|_2 \stackrel{\mathrm{def}}{=} \sqrt{\tilde{r}(A^t A)},
$$
where $\tilde{r}$ is the spectral radius. Additionally we denote $\mathcal{B}_{3\rho}(y) \stackrel{\mathrm{def}}{=} \chi_y^{-1}(B_{3\rho}(\chi_y(y))$.

We say that an atlas $\mathfrak{W} = \{ (W_y,\chi_y) \}_{y\in M}$ is $(\rho, \vartheta)$--technical, $\rho, \vartheta > 0$, if
\begin{enumerate}
\item[\textit{\textbf{W3}}] $\forall y \in M \Rightarrow \mathcal{B}_{3\rho}(y) \subset W_y,$
\item[\textit{\textbf{W4}}] For every chart $(W,\chi) \in \mathfrak{W}$, $C^{0,1}$-norm\, $\textbf{G}\circ \chi^{-1}$\, does not exceed $\vartheta$, and $L_\infty$-norm of \,$|\textbf{G}\circ \chi^{-1} |_2$\, can be estimated from below by $\vartheta^{-1}$. 
\item[\textit{\textbf{W5}}] There exists a finite subatlas $\mathcal{U}$ of the atlas for $M$ such that for all the transition functions from $U \in \mathcal{U}$ into $W \in \mathfrak{W}$ and their inverses $C^{0,1}$-norm of the Jacobi matrix $\left( \frac{\partial x^{i'}}{\partial x^i}\right)$ does not exceed $\vartheta$.
\end{enumerate}

\begin{definition}
An open subset $\Omega \subset M$ satisfies the uniform $\omega$--cusp condition with parameters $(r,\vartheta)$ if there exists $(\psi(r), \vartheta)$--technical atlas $\mathfrak{W} = \{(W_y, \chi_y)\}_{y \in M}$ such that for arbitrary $y \in M$ the open set $\chi_y(\Omega \cap \mathcal{B}_{3\psi(r)}(y))$ satisfies the $\omega$--cusp condition with parameter $r$ at the point $\chi_y(y)$. This class will be denoted $\mathcal{W}^{\omega}_{r,\vartheta}$.
\end{definition}

\begin{definition}
\label{comega}
Boundary $\partial \Omega$ of a domain $\Omega \subset M$ is of class $C^{0,\omega}$ if there exists such a subatlas $\mathcal{U} = \{(U, \kappa_U) \}$ for $M$ that nonempty $\kappa_U (\partial \Omega \cap U)$ can be represented by a graph of continuous function $g_U$ with modulus of continuity not exceeding $C_U \omega$, $C_U \in \mathbb{R}_+$, $\omega(0) = 0$, where the intersection of $\kappa_U(\Omega \cap U)$ and off-graph is empty.
\end{definition}

Domains $\Omega \subset M$ with boundaries locally representable as graphs of continuous functions we call domains of $C$--class. In the case of compact $M$ by virtue of Cantor theorem one has: for any domain $\Omega \subset M$ there exists such a positive semi-additive function $\omega_\Omega$, $\omega_\Omega(0) = 0$, that $\partial \Omega \in C^{0,\omega_\Omega}$. 

\begin{proposition}
\label{classif}
For arbitrary domain $\Omega \subset M$ with boundary of the class $C^{0,\omega_\Omega}$, $\omega_\Omega(0) = 0$, there exists a class $\mathcal{W}^{C \omega_\Omega}_{r,\vartheta}\ni \Omega$, where $C \equiv \mathrm{const} > 0$.
\end{proposition}

\begin{proof}
Since $M$ is compact we may assume that for a fixed domain $\Omega \subset M$ with boundary of the class $C^{0,\omega_\Omega}$, an atlas $\mathcal{U}$ (in the def \ref{comega}) is finite, sets $\kappa_U(U) \subset \mathbb{R}^d$ are bounded, and the mappings $\kappa_U$ are defined on $\tilde{U} \supset \bar{U}$. We select a new atlas $\{(U', \kappa_{U}) \}$ with the property $U' \Subset U$ and set $r = \psi^{-1}\left[\frac{1}{100} \min_{U} \mathrm{dist}( \partial \kappa_U U, \kappa_U U')\right]$, where
$$
\mathrm{dist} (A,B) = \inf_{a \in A, b \in B} |a- b|.
$$
For every $y \in M$ we choose $U' \in \mathcal{U}$, $y \in U'$, and set $(W_y, \chi_y) = (U', \kappa_U)$. Then there exists such a number $\vartheta > 0$, that $\mathfrak{W} = \{ (W_y, \chi_y) \}_{y\in M}$ is a $(\psi(r), \vartheta)$--technical atlas. In fact, condition \textbf{\textit{(W3)}} is satisfied according to construction while \textbf{\textit{(W4)}} and \textbf{\textit{(W5)}} are fulfilled in view of all the mappings $\chi_y$ are obtained as restrictions of finite number of coordinate diffeomorphisms.

Condition \textbf{\textit{(W2)}} is also valid for $\omega = C \omega_\Omega$, $C = \max_U C_U$, since images of $\Omega$ under coordinate diffeomorphism $\chi_y$ are located no one side with respect to graph of $g_y$.
\end{proof}

Similar claim in the case of Lipschitz boundary can be found in \cite{ChenaisEng}. It should be noted that $\omega \equiv 0$ if $\omega(h) = o(h)$. Nevertheless in the case of manifolds the class $C^{0,0(\cdot)}$ can be nonempty.

\section{Necessary background, constructions and estimates.}

\subsection{Hausdorff convergence.}
\label{Haus_p}

Let $X, Y$ be an arbitrary subsets of a connected metric compact space $(M,d)$. Set $d(x,Y) = \inf_{y\in Y}d(x,y)$ and consider the function $e(X,Y) = \sup_{x \in X} d(x, Y)$. If $X\backslash Y \neq \varnothing$ then
$$
e(X,Y) = \sup_{x \in X\backslash Y} d(x,\partial Y) = e(X\backslash Y, \partial Y).
$$
Next introduce $\check{e}(X,Y) = e(M \backslash Y, M \backslash X)$. In view of $(M \backslash Y) \backslash (M \backslash X) = X \backslash Y$ and $\partial X = \partial (M \backslash X)$ one has
\begin{equation}
\check{e}(X,Y) = \sup_{y \in M \backslash Y} d(y, M \backslash X) = \sup_{x \in X\backslash Y} d(x,\partial X) = e(X\backslash Y, \partial X).
\end{equation}
In terms of blowing $X^{\varepsilon} = \left\{ x\in M\left| \,\,d(x,X) < \varepsilon \right.\right\}$, $\varepsilon > 0$, and contraction $X^{-\varepsilon} = \left\{ x \in X\left| \,\, \right.\right.$ $ \left.\left. \forall z \in M\!\!: \, d(x,z) < \varepsilon \Rightarrow z \in X \right.\right\}$ the basic functions $e$ and $\check{e}$ can be described as follows
$$
e(X,Y) = \inf\left\{ \varepsilon > 0\left| X \subset Y^\varepsilon \right.\right\} ,\,\,\,\, \check{e}(X,Y) = \inf\left\{ \varepsilon > 0\left| X^{-\varepsilon} \subset Y \right.\right\}.
$$
It follows that Hausdorff distance functions can be introduced by the formulas:
\begin{align}
\label{HausClose}
d_{\mathcal{H}}(X, Y) = \max \left\{ e(X,Y), e(Y,X) \right\};\\
\label{HausOpen}
d^{\mathcal{H}}(X, Y) = \max \left\{ \check{e}(X,Y), \check{e}(Y,X) \right\}.
\end{align}
These functions are metrics on the families of closed and open set respectively. It is useful to consider stronger version of Hausdorff metric, namely upper Hausdorff--Pompeiu distance
\begin{equation}
\label{HausFull}
d^{{\mathcal{HP}}}(X,Y) = \max\left\{ d_{\mathcal{H}}(X,Y), d^{\mathcal{H}}(X,Y) \right\}.
\end{equation}

\begin{remark}
Notice fundamental difference between (\ref{HausFull}) and (\ref{HausOpen}). The family of all open subsets in a fixed metric compact space $K$ is compact w.r.t. (\ref{HausOpen}) according to Blaschke theorem (see \cite{Burago}) but w.r.t. (\ref{HausFull}) this family loses compactness property though completeness remains. To see the latter it is sufficient to consider subgraphs of the functions $(2 + \sin nx)$ considered on the closed interval $[0, \pi]$.
\end{remark}

For our purposes (see th. \ref{th_42}) the following minimum of all the distances of the Hausdorff type will be necessary:
$$
d_{{\mathcal{HS}}}(X,Y) = \min\left\{ e(X \Delta Y, \partial Y), e(X \Delta Y, \partial X) ,d_{\mathcal{H}}(X,Y), d^{\mathcal{H}}(X,Y) \right\}
$$

Quantities $e(X,Y)$ and $\check{e}(X,Y)$ give us four nonequivalent ways to measure distances, thus $d_{{\mathcal{HS}}}$ is the weakest quantity defining convergence of sets among those that can be constructed by means of $e(X,Y)$, $\check{e}(X,Y)$, $e(Y,X)$, $\check{e}(Y,X)$.

\subsection{Estimates of distances between solutions.}

For a Hilbert space $V$ and its closed subspaces $V_1$ and $V_2$. Consider the problems:
$$
u_i \in V_i, \,\,\,\,\, \Phi(u_i,v_i) = \langle f, v_i \rangle\,\,\,\,\,\, \forall v_i \in V_i,
$$
where $f\in V'$, and $\Phi$ is a bilinear continuous function on $V$ possessing positive $\alpha, \beta \in \mathbb{R}$ such that
$$
\alpha \|u\|_V^2 \leq \Phi (u,u) \leq \beta \| u\|_V^2\,\,\,\,\,\, \forall u\in V.
$$
By Lax--Milgram lemma solutions $u_i = \mathcal{G}(f;V_i)$ exist and are unique. All the statements of this subsections can be found in  \cite{Savare}. To formulate the following lemma and its corollary we need the standard notation $d_V(v, A)$ for the distance between $v\in V$ and a subset $A \subset V$.
\begin{lemma}
\label{lem_32}
For solutions $u_1$ and $u_2$ the following inequality holds:
\begin{equation}
\label{eq_8}
\| u_1 - u_2 \|_V \leq \sqrt{\frac{\beta}{\alpha}} \left( d_V (u_1, V_1 \cap V_2) + d_V (u_2, V_1 \cap V_2) \right).
\end{equation}
Moreover, if $V^{1,2}$ is a closed subspace, containing $V_1 \cap V_2$ then for $u^{1,2} = \mathcal{G}(f;V^{1,2})$ the inequality
\begin{equation}
\label{eq_9}
\| u_1 - u_2\|_V \leq \sqrt{\frac{\beta}{\alpha}} \left(d_V(u^{1,2}, V_1) + d_V(u^{1,2}, V_2)\right)
\end{equation}
takes place.
\end{lemma}

\begin{corollary}
\label{cor_33}
The estimate (\ref{eq_9}) takes the form
$$
\| u_1 - u_2\|_V \leq \frac{\beta}{\alpha} \left(d_V(u^{1,2}, V_1) + d_V(u^{2,1}, V_2)\right)
$$
where $u^{2,1} = \mathcal{G}(f;V^{2,1})$ and $V^{2,1}$ contains $V_1 \cup V_2$.
\end{corollary}

\subsection{Lemma about perturbation of eigenvalues.}

Let $O \subsetneq (M \backslash \partial M)$ be open non-void and $M_O \stackrel{\mathrm{def}}{=} M \backslash \overline{O}$. The assumption $\Omega \subset M_O$ will be assumed for every domain considered below. Let $\mathrm{p} = \mathrm{p}_{M_O}$ denote Friedrichs constant (i.e. $\inf\left\{\mathrm{p} \, \left| \, \| u \|^2_{L_2(M_O)} \leq \mathrm{p} \,\,\right.\| u \|^2_{\SobSl{M_O}} \right\}$) and
\begin{equation}
\label{norms}
\| u \|^2_{V_\Omega} \stackrel{\mathrm{def}}{=} \int_{\Omega} \textbf{A} (\nabla u , \nabla u) \, d\mu, \,\,\,\,\,  \|u\|^2_{L_\Omega} \stackrel{\mathrm{def}}{=} \mathrm{p} \int_\Omega |u|^2 d\mu,
\end{equation}
be the norms in the spaces $V_\Omega = \Sob(\Omega)$ and $L_\Omega = L_2(\Omega)$ respectively with the standard norms in these spaces defined by the formulae:
\begin{equation}
\label{norms_}
\| u \|^2_{\SobSl(\Omega)} \stackrel{\mathrm{def}}{=} \int_{\Omega} \textbf{G} (\nabla u , \nabla u)  d\mu; \,\,\,\,\,\,  \|u\|^2_{L_2(\Omega)} \stackrel{\mathrm{def}}{=} \int_\Omega |u|^2 d\mu.
\end{equation}
In addition, we denote $V = V_{M_O}$, $L = L_{M_O}$.

Again by Lax-Milgram lemma it follows that problem (\ref{ProbDir}) possesses unique colution $u_f \stackrel{\mathrm{def}}{=} \mathcal{G}(f; V_\Omega) \in V_\Omega$ for arbitrary $f \in V'$; thus
\begin{equation}
\label{Lax}
\| \mathcal{G}(f;V_\Omega) \|_V \leq \alpha^{-1} \| f \|_{V'} \,\,\,\, \forall f\in V'.
\end{equation}

Let pairs $(u_n^{(1)}, \lambda_n^{(1)})$, $(u_n^{(2)}, \lambda_n^{(2)})$ be solutions to the eigenvalue problem (\ref{SpecDir}) for domains $\Omega_1$ and $\Omega_2$ respectively; enumeration $\{\lambda^{(i)}_n\}$ in ascending order is meant. Denote by $P_{\Omega_1}: V \to V_{\Omega_1}$ the orthogonal projection $V$ onto $V_{\Omega_1}$ and set $S_n^{(1)} = \mathrm{span} ( u_1^{(1)}, \ldots, u_n^{(1)} )$. 

\begin{lemma}[cf. \cite{Birkhoff}]
\label{Birkhoff}
Fix $n\in\mathbb{N}$ and assume there are positive numbers $A_n$ and $B_n < \mathrm{p}$ such that for every $u \in S_n^{(1)}$ the inequalities
\begin{equation}
\label{proj}
\| P_{\Omega_1}u - u \|_{V}^2 \leq A_n \|u\|_{L}^2, \,\,\,\,\,\, \| P_{\Omega_1}u - u \|_{L}^2 \leq B_n \|u\|_{L}^2,
\end{equation}
hold; then
$$
 \lambda_n^{(1)} \geq \lambda_n^{(2)} - \frac{A_n}{(\sqrt{\mathrm{p}} - \sqrt{B_n})^2 \mathrm{p}}.
$$
\end{lemma}

\subsection{Local estimate necessary for resolvent convergence}

Arguments in this section are based on a generalization of the technique proposed in \cite{Savare}. We start with the following

\begin{lemma}[cf. \cite{EvansGar} p. 24, \cite{Ambrosio} p. 49]
\label{Bez}
Let $M$ be a compact manifold, $d$ a metric associated with a Riemannian structure on $M$, $\mathcal{F}$ a family of opened balls with $\inf_{B\in\mathcal{F}} \mathrm{diam} B > 0$ and $A$ the set of its centers. Then there exists a finite subfamily $\{B_{r_j}(a_j)\}_{j=1}^{J} \subset \mathcal{F}$ such that
$$
A \subset \cup_{j=1}^{J} B_{r_j}(a_j),
$$
with $\{B_{r_j/3}(a_j)\}_{j=1}^{J}$ being disjoint.
\end{lemma}

\begin{proposition}
\label{prop_36}
Let $X,Y \subset M_O$ be open, atlas $\mathfrak{W}$ be $(\rho,\vartheta)$--technical and $v \in V_Y$. Assume that for every $y \in \Lambda= Y\backslash X$ there exists a vector $\nu(y)$, such that
$$
x\in \chi_y(\mathcal{B}_\rho(y)\backslash X) \Rightarrow (x + \nu(y)) \notin \chi_y\left( Y\cap \mathcal{B}_{3\rho}(y)\right),
$$
and a function $H \in L_1(M)$ satisfies the inequality
$$
\| v^y_{\nu(y)}  - v  \|^2_{V_{\mathcal{B}_{\rho}(y)}} \leq \| H \|_{L_1(\mathcal{B}_{3 \rho}(y))},
$$
for every $y \in \Lambda^{3\rho} = \cup_{y \in \Lambda} \mathcal{B}_{3\rho} (y)$, where $v^y_h = v \circ \chi_y^{-1} \circ (x + h) \circ \chi_y$. Then there exists a function $w \in V_X$ (independent on the choice of $H$) such that
$$
\| w  - v \|_{V}^2 \leq C \| H \|_{L_1(\Lambda^{3 \rho})},
$$
where $C = C(M_O, \vartheta, \rho, \mathcal{A}) \equiv \mathrm{const}$.
\end{proposition}

\begin{proof}
Let $\mathcal{O}_{R}(y) = \{ z \in M\, | \, \delta(z,y) < R \}$ be the geodesic ball with respect to the metric associated with the Riemannian structure $g$. Applying lemma \ref{Bez} to the family  $\mathcal{F} = \{ \mathcal{O}_{\frac{\rho}{2 \vartheta}} (y) \}_{y \in \Lambda}$ and the set $A = \Lambda$, we get the finite set  $\{x_j\}_{j=1}^{J}$ such that
$$
i \neq j \Rightarrow \mathcal{O}_{\frac{\rho}{6\vartheta}}(x_j) \cap \mathcal{O}_{\frac{\rho}{6\vartheta}}(x_i) = \varnothing,\,\,\,\,\,\Lambda \subset \cup_{j\in J} \mathcal{O}_{\frac{\rho}{2 \vartheta}}(x_j) \subset \cup_{j\in J} \mathcal{B}_{\frac{\rho}{2}}(x_j) ,
$$
$$
\mathcal{B}_{\frac{\rho}{6 \vartheta^2}}(x_i) \cap \mathcal{B}_{\frac{\rho }{6 \vartheta^2}}(x_j) = \varnothing.
$$
Hence
\begin{equation}
\label{eq_star}
J \leq \frac{\mu(M)}{\min_{j} \mu(\mathcal{O}_{\frac{\rho}{6\vartheta}}(x_j))} \leq \frac{\mu(M)}{\inf_{z\in M}\mu(\mathcal{O}_{\frac{\rho}{6\vartheta}}(z))},
\end{equation}
and by virtue of $M$ compact and measure $\mu$ absolutely continuous there exists a point $z\in M$ where (positive) $\mathrm{inf}$ in (\ref{eq_star}) is reached; thus  $J < C_2(M, \vartheta, \rho)$.

The required function $w$ will be built by means of a certain partition of unity. To construct it we introduce the functions
$$
\varphi_j(x) = \min \{1, (3 - 3|\chi_{x_j}(x) - \chi_{x_j}(x_j)| / \rho)^+\}, \,\,\,\, \varphi_0(x) = \min\{ 1, 6\, \delta (x, \Lambda) / \rho\}.
$$
and $\mathfrak{g}(x,\xi) = \textbf{G}_x(\xi,\xi)$. Then one has $0 \leq \varphi_j(x)\leq 1$ and
$$
\mathrm{supp}(\varphi_j) \subset \overline{\mathcal{B}_{\rho}(x_j)},\,\,\,\, \varphi_j|_{\mathcal{B}_{3\rho/4}(x_j)} \equiv 1, \,\,\,\, \mathfrak{g}(x, \nabla \phi_j)^{1/2} \leq C_3(M) \rho^{-1} \textbf{1}_{\mathcal{B}_{\rho}(x_j)}(x),\,\,\,\, j \in \mathbb{Z}_+ \cap [0, J],
$$
where $\textbf{1}_A(x) = 1$ is the indicator of $A \subset M$; thus $1 \leq \sum_{j =0}^{J} \varphi_j \leq 1 + J$. The function $\varkappa_j = \frac{\varphi_j}{\sum_{k = 0}^{J} \varphi_k}$ possess the following properties:
$$
\sum_{ j = 0}^{J} \varkappa_j \equiv 1, \,\,\,\,\, 0\leq \varkappa_j \leq 1, \,\,\, \,\,\mathfrak{g}(x, \nabla \varkappa_j )^{1/2} \leq C_4(M,\vartheta, \rho)/\rho = \frac{C_2(M, \vartheta, \rho)C_3(M)}{\rho}.
$$
If  $w = \sum_{j=0}^{J} \varkappa_j \tilde{v}_j$, $\tilde{v}_j = \textbf{1}_{\mathcal{B}_{2\rho}(x_j)} v^{x_j}_{\nu(x_j)}$, then 
\begin{align*}
\| v - w \|^2_{L} \leq \mathrm{p} \int_M \left|\sum_{j = 0}^{J} \varkappa_j (v - \tilde{v}_j)\right|^2 d\mu \leq \sum_{j = 0}^{J} \mathrm{p} \int_{\mathcal{B}_{\rho}(x_j)} |v - v^{x_j}_{\nu(x_j)}|^2 d\mu = \\
\sum_{j=0}^{J} \| v - v^{x_j}_{\nu(x_j)} \|^2_{L_{\mathcal{B}_{\rho}(x_j)}} \leq \sum_{j=0}^{J} \| v - v^{x_j}_{\nu(x_j)} \|^2_{V_{\mathcal{B}_{\rho}(x_j)}} \leq C_2(M, \vartheta, \rho) \int_{\Lambda^{3\rho}} H\, d\mu,
\end{align*}
Using the symbol $\mathfrak{a}(x,\xi) = \textbf{A}_x(\xi,\xi)$ of operator $\mathcal{A}$ one finally has
\begin{align*}
\| v - w \|^2_{V} = \int_M \mathfrak{a} \left( x, \nabla \sum_{j=0}^{J} \varkappa_j (v - \tilde{v}_j) \right) d\mu \leq \\
2C_2(M, \vartheta, \rho) \sum_{j=0}^{J} \int_M \mathfrak{a}(x,\nabla \varkappa_j) |v - \tilde{v}_j|^2  d\mu + 2\sum_{j = 0}^{J} \int_M \varkappa_j \mathfrak{a}(x, \nabla v - \nabla \tilde{v}_j) d\mu \leq\\
\leq 2 C_4(M, \vartheta, \rho) \sum_{j=0}^{J} \int_{\mathcal{B}_{3\rho}(x_j)} \left(  \frac{C_5(\mathcal{A})}{\mathrm{p}\rho^2} H + H \right) \,d\mu \leq C \| H \|_{L_1(\Lambda^{3 \rho})}.
\end{align*}

\end{proof}

\subsection{Estimates in domains from $\mathcal{W}^{\omega}_{r, \vartheta}$.}

\begin{lemma}[see \cite{Brezis}]
\label{lem_37}
If $v \in H^1(\mathbb{R}^d)$, $x_0 \in \mathbb{R}^d$, the for any $h\in \mathbb{R}^d$, $|h| < \rho,$
$$
\int_{B_{2\rho}(x_0)} |v(x+h) - v(x)|^2 dx \leq |h|^2 \int_{B_{3\rho}(x_0)} |\nabla v(x)|^2 dx
$$
\end{lemma}
The following claim allows one to find a function $H$ satisfying condition in Proposition~\ref{prop_36}. 

\begin{theorem}
\label{th_38}
Let $Z\subset M$, $Z \in \mathcal{W}^{\omega}_{r,\vartheta}$, $f\in L_2(M)$, $u = \mathcal{G}(f;V_Z)$, $y \in M$, $h = |h| \xi_y$, $|h| < \psi(r)$. Then for every positive $\rho < \psi(r)$ there exists such a number $\tilde{C} = \tilde{C}(M_O, r, \vartheta, \mathcal{A})$ that
$$
\|u - u^y_h\|^2_{V_{\mathcal{B}_{\rho}(y)}} \leq \tilde{C} \|u\|_{V_{\mathcal{B}_{3 \rho}(y)}} \left[  \|u\|_{V_{\mathcal{B}_{3 \rho}(y)}} + \| f \|_{L_{Z \cap \mathcal{B}_{2\rho}(y)}} \right] \cdot |h|.
$$
\end{theorem}

\begin{proof}
We need the function
$$
\varkappa_y = \min\left\{ 1, \left( 2 - \frac{|x-y|}{\rho}\right)^+ \right\}\circ \chi_y,
$$
it possesses the following properties:
$$
0 \leq \varkappa_y \leq 1, \,\,\, \mathfrak{g}(x,\nabla \varkappa_y) \leq C_3(M)  \rho^{-1}, \,\,\, \varkappa_y|_{\mathcal{B}_\rho(y)} \equiv 1.
$$
For $v\in H^1(\mathcal{B}_{3\rho}(y))$ we introduce notation
$$
\mathcal{T}_{y,h} v = (1 - \varkappa_y) v + \varkappa_y v^y_h, 
$$
and remark that $\mathcal{T}_{y,h} v - v = \varkappa_y \left( v^y_h - v\right)$. Since $Z \in \mathcal{W}^{\omega}_{r,\vartheta}$, one has
\begin{equation}
\label{eq_15}
v\in V_Z \Rightarrow \mathcal{T}_{y,h} v \in V_Z \subset V,\,\,\, \mathrm{supp}(v), \,\mathrm{supp}(\mathcal{T}_{y,h} v) \subset \bar{Z},
\end{equation}
and hence
\begin{equation}
\label{eq_16}
 \| u^y_h - u \|^2_{V_{\mathcal{B}_{\psi(r)}(y)}} =  \| \mathcal{T}_{y,h} u - u \|^2_{V_{\mathcal{B}_{\psi(r)}(y)}} \leq \Phi(\mathcal{T}_{y,h} u, \mathcal{T}_{y,h} u) - \Phi(u,u) + 2 \langle f, u - \mathcal{T}_{y,h} u \rangle.
\end{equation}
The last summand can be estimated by lemma \ref{lem_37}:
$$
2 \langle f, u - \mathcal{T}_{y,h} u \rangle \leq \vartheta^3\,  |h|\, \| u \|_{H^1(\mathcal{B}_{3\rho})} \|f\|_{L_2(Z \cap \mathcal{B}_{3\rho}(y))} \leq \vartheta^3 \mathrm{p}^{-1} \alpha^{-1}  |h| \| u \|_{V_{\mathcal{B}_{3\rho}}} \| f \|_{L_{Z \cap \mathcal{B}_{3\rho}(y)}}.
$$
To estimate the remaining term in (\ref{eq_16}) we use the formula
$$
\nabla (\mathcal{T}_{y,h} v) = \varkappa_y \nabla (v^y_h) + (1 - \varkappa_y) \nabla v + \nabla \varkappa_y ( v^y_h - v) = \mathcal{T}_{y,h} \nabla v + \nabla \varkappa_y (v_h^y - v),
$$
and property (\ref{eq_15}) implies that $\Phi(\mathcal{T}_{y,h} u, \mathcal{T}_{y,h} u) - \Phi(u,u)$ does not exceed
\begin{align}
\label{eq_17}
& \int_Z \mathfrak{a}(y, \mathcal{T}_{y,h} \nabla u + \nabla \varkappa_y (u^y_h - u)) d\mu  - \int_Z \mathfrak{a}(y, \mathcal{T}_{y,h} \nabla u ) d\mu + \\
\label{eq_18}
& \int_Z \mathfrak{a}(y, \mathcal{T}_{y,h} \nabla u ) d\mu  - \int_Z \mathfrak{a}(y, \nabla u ) d\mu.
\end{align}
Since for $\xi, \eta \in \mathbb{R}^d$ the following inequality
$$
\mathfrak{a}(x, \xi + \eta) - \mathfrak{a}(x, \xi) \leq \left( \mathfrak{a}(x, \eta) \mathfrak{a}(x, 2\xi + \eta) \right)^{1/2} \leq \mathfrak{a}(x,\eta)^{1/2} (2 \mathfrak{a}(x,\xi)^{1/2} + \mathfrak{a}(x,\eta)^{1/2}),
$$
holds, we can apply it for $\xi = \mathcal{T}_{y,h} \nabla u$, $\eta = \nabla \varkappa_y (u^y_h - u))$, hence integrals in line (\ref{eq_17}) can be estimated as follows
\begin{align*}
 \int_Z \mathfrak{a}(y, \mathcal{T}_{y,h} \nabla u + \nabla \varkappa_y (u^y_h - u)) d\mu  - \int_Z \mathfrak{a}(y, \mathcal{T}_{y,h} \nabla u ) d\mu \leq \\
 \int_Z \mathfrak{a}(y, \mathcal{T}_{y,h} \nabla u + \nabla \varkappa_y (u^y_h - u)) d\mu  - \int_{Z\cap \mathcal{B}_{2\rho}(y)} \mathfrak{a}(y, \mathcal{T}_{y,h} \nabla u ) d\mu  \leq\\
\int_Z \mathfrak{a}\left(y, \nabla \varkappa_y (u^y_h - u)\right)^{1/2} \left[ \left( \mathfrak{a}(y, \nabla \varkappa_y (u^y_h - u)) \right)^{1/2} + 2 \left( \mathfrak{a}(y,\mathcal{T}_{y,h} \nabla u ) \right)^{1/2} \right] d\mu \leq\\
\left( \int_Z \mathfrak{a}(y, \eta)d\mu \right)^{1/2} \cdot \left[ \left( \int_Z \mathfrak{a}(y, \eta) d\mu \right)^{1/2} + 2 \left( \int_Z \mathfrak{a}(y,\xi )  d\mu \right)^{1/2} \right]\leq \\
\tilde{C}_3(M, \vartheta, \mathcal{A}) \| u -  u^y_h \|_{L_{\mathcal{B}_{2\rho}(y)}} \left( \tilde{C}_3(M, \vartheta, \mathcal{A}) \| u - u^y_h \|_{L_{\mathcal{B}_{2\rho}(y)}} + 2 \| \mathfrak{a}(x,\mathcal{T}_{y,h} \nabla u) \|_{L_{\mathcal{B}_{2\rho}(y)}} \right).
\end{align*}
It follows from definition of $\mathcal{T}_{y,h}$ that
$$
\| \mathfrak{a}(x,\mathcal{T}_{y,h} \nabla u) \|_{L_2 (\mathcal{B}_{2\rho}(y))} \leq  2 \tilde{C}_4(M, \rho) \| u \|_{V_{\mathcal{B}_{3\rho}(y)}}.
$$
Therefore applying lemma \ref{lem_37} one has
\begin{align*}
\int_Z \mathfrak{a}(y, \mathcal{T}_{y,h} \nabla u + \nabla \varkappa_y (u^y_h - u)) d\mu  &- \int_Z \mathfrak{a}(y, \mathcal{T}_{y,h} \nabla u ) d\mu \leq\\
2 \tilde{C}_3(M, \vartheta, \mathcal{A}) \left(\tilde{C}_3(M, \vartheta, \mathcal{A}) +  \tilde{C}_4(M, \rho)\right) \| u \|^2_{V_{\mathcal{B}_{3\rho}(y)}} |h| &= \tilde{C}_5(M, \vartheta, \rho, \mathcal{A}) \| u \|^2_{V_{\mathcal{B}_{3\rho}(y)}} |h|.
\end{align*}
The integrals in (\ref{eq_18}) can be estimated with regards for convexity of $\mathfrak{a}$:
\begin{align*}
&\mathfrak{a} (x, \mathcal{T}_{y,h} \nabla v) - \mathfrak{a} (x, \nabla v) \leq \\
(1-\varkappa_y ) \mathfrak{a}(x,\nabla v) + \varkappa_y \mathfrak{a}(x, &\nabla (v^y_h)) - \mathfrak{a}(x, \nabla v) = 
\varkappa_y \left[ \mathfrak{a}(x, \nabla v^y_h )  - \mathfrak{a}(x, \nabla v) \right],
\end{align*}
where $x \in \mathrm{supp}\kappa_y$. Hence we get
\begin{align*}
\int_Z \mathfrak{a}(y, \mathcal{T}_{y,h} \nabla u ) d\mu  - \int_Z \mathfrak{a}(y, \nabla u ) d\mu \leq \int_{Z \cap \mathcal{B}_{ 2\rho}(y)} \varkappa_y \left[ \mathfrak{a}(x, \nabla (v^y_h))  - \mathfrak{a}(x, \nabla v) \right] d\mu = \\
\int_{Z \cap \chi^{-1}_y (B_{2\rho}(y) + h)} (\varkappa_y)^y_{-h} \mathfrak{a}( x - h, \nabla v) d (\mu^y_{-h}) - \int_{Z\cap \mathcal{B}_{2\rho}(y)} \varkappa_y \mathfrak{a} (x, \nabla v)d\mu \leq \\
\int_{Z \cap \mathcal{B}_{3\rho}(y)} (\varkappa_y)^{y}_{-h} (\mathfrak{a}(x - h,\nabla v) - \varkappa_y \mathfrak{a}(x - h,\nabla v)) d (\mu^y_{-h})  +\\
+ \int_{Z \cap \mathcal{B}_{3\rho}(y)} \varkappa_y \mathfrak{a}(x - h,\nabla v)) ( d \mu^y_{-h} - d\mu)+ \int_{Z\cap \mathcal{B}_{2\rho}(y)} \varkappa_y (\mathfrak{a}(x - h,\nabla v) - \mathfrak{a}(x,\nabla v)) d\mu  \leq\\
\tilde{C}_6(M_O, \vartheta, \rho, \mathcal{A}) |h| \cdot \| u \|_{V_{\mathcal{B}_{3\rho}(y)}}.
\end{align*}
\end{proof}

\subsection{Properties of the sets satisfying uniform $\omega$-cusp condition}

\subsubsection{Local geometry}

In the proposition of this subsection $\Omega \subset M \subset \mathbb{R}^d$ denotes a domain of the class $\mathcal{W}^{\omega}_{r,1}$; notation $\Omega^\varepsilon$ see in subsection \ref{Haus_p}.

\begin{lemma}
\label{lem_39}
Let $\Omega$ satisfy uniform  $\omega$-cusp condition with parameter $r$ at point $x$. Then for any postitve $\varepsilon \leq \rho = \psi(r)$ the set $\Omega^{\varepsilon}$ satisfies uniform $\omega$-cusp condition with parameter $r_2 = \psi^{-1}\left( \psi(r) / 2 \right)$ at the point $x$.
\end{lemma}
\begin{proof}
Choose $h \in \mathcal{C}_{\omega, r}(\xi_x)$ and $y\in B_{3\rho/2}(x)$; if $y\in\Omega^\varepsilon$ then there exists $z\in \Omega$ such that $|z-y| < \varepsilon \leq \rho$. Since $|z-x|<\rho + \frac{3}{2}\rho < 3\rho$ then from the local definition \ref{W_def1} one has
$$
z - h\in \Omega, \,\,\, \mathrm{dist}(y - h, \Omega) \leq |y - h - (z - h)| = |y - z| < \varepsilon,
$$
that is $y - h \in \Omega^\varepsilon$.
\end{proof}

\begin{remark}
\label{rem_310}
By simple geometric reason it is easy to see that
$$
0 < \varepsilon \leq \phi^{-1}\left( \frac{\psi(r)}{2} \right) \Rightarrow B_\varepsilon \left( \phi(\varepsilon) \xi \right) \subset \mathcal{C}_{\omega, r}(\xi).
$$
\end{remark}

In fact, for sufficiently small $\varepsilon$ the inclusion
$$
B_\varepsilon \left( (\omega(\varepsilon) + \varepsilon) \xi \right) \subset \mathcal{S}_{\omega, r}(\xi) \cup \mathcal{F}_{\omega, r}(\xi).
$$
holds. As regards sufficiency it is ensured by the inequality
$$
2 \phi(\varepsilon) < \psi(r).
$$

\begin{lemma}
\label{lem_tran}
Let $Z$ satisfies uniform $\omega$-cusp condition with parameter $R$ at the point $y$, $\rho = \psi(R)$, $\omega(0) = 0$, and
$$
- \phi^{-1}\left( \frac{\psi(R)}{4} \right) \leq \eta \leq 0 \leq \varepsilon \leq \phi^{-1}\left( \frac{\psi(R)}{4} \right)
$$
Then
$$
x \in B_{2\rho} (y) \backslash Z^{\eta} \,\,\Rightarrow \,\,x + \left[\phi(\varepsilon) + \phi (-\eta)\right] \xi_y \notin Z^{\varepsilon}.
$$
\end{lemma}

\begin{proof} 
Let us denote $z = x + \left[\phi(\varepsilon) + \phi (-\eta)\right] \xi_y$, $t = x + \phi (-\eta) \xi_y$; assuming $t \in Z$ because of inclusion $t \in B_{3\rho}(y)$ one has by Remark \ref{rem_310} that
$$
B_{-\eta}(x) = B_{-\eta} \left( t - \phi(-\eta) \xi_y \right) \subset t - \mathcal{C}_{\omega, r}(\xi_y) \subset Z,
$$
that is $x \in Z^{\eta}$. This contradiction means that $t \not\in Z$. Applying Remark  \ref{rem_310} again one has 
$$
B_\varepsilon(z) = B_\varepsilon(t + \phi(\varepsilon)\xi_y) \subset  t +  \mathcal{C}_{\omega, r}(\xi_y) \subset \mathbb{R}^d \backslash Z,
$$
hence $\mathrm{dist}(z, Z) \geq \varepsilon$.
\end{proof}

\begin{lemma}
\label{lem_312}
Let $\Omega_1, \Omega_2$ be open subsets of $\mathbb{R}^d$, $\Omega_1 \in \mathcal{W}^{\omega}_{r,\vartheta}$, $\rho = \psi(r)$, $\omega(0) = 0$, $e(\Omega_2, \Omega_1) \leq \phi^{-1}\left( \frac{\psi(r)}{2} \right)$. Then $\check{e}(\Omega_2, \Omega_1) \leq \phi(e(\Omega_2, \Omega_1))$.

If in addition $\Omega_2 \in \mathcal{W}^{\omega}_{r, \vartheta}$ and  $d_H(\Omega_1, \Omega_2) \leq \phi^{-1}\left( \frac{\psi(r)}{2} \right)$ one has
$$
d^\mathcal{HP}(\Omega_1, \Omega_2) \leq \phi (d_\mathcal{HS}(\Omega_1, \Omega_2)).
$$
\end{lemma}

\begin{proof}
It is sufficient to prove the first statement; the second follows from definition of $\check{e}(X,Y)$ and the fact that locally $M \backslash \bar{\Omega}$ satisfies $\omega$-cusp condition if $\Omega$ satisfies this condition.

If $\lambda = e(\Omega_1, \Omega_2)$ then $\Omega_2 \subset \Omega_1^{\lambda}$. For $y \in \Omega_2 \backslash \Omega_1$ one has
$$
y + \phi(\lambda) \xi_y \in \mathbb{R}^d \backslash \Omega_1^{\lambda} \subset \mathbb{R}^d \backslash \Omega_2.
$$
by lemma \ref{lem_tran} and inequality $\lambda \leq \phi^{-1}\left( \frac{\psi(r)}{2} \right)$. It follows that $\mathrm{dist}(y, \mathbb{R}^d \backslash \Omega_2) \leq \phi(\lambda)$ and hence
$$
\check{e}(\Omega_2, \Omega_1) = \sup_{y\in \Omega_2 \backslash \Omega_1} d(y, \mathbb{R}^d \backslash \Omega_2) \leq \phi (\lambda).
$$
\end{proof}

\subsubsection{Global analogs of lemmas \ref{lem_39} and \ref{lem_tran}.}

\begin{corollary}
\label{obvious}
If  $\Omega \in \mathcal{W}^{\omega}_{r,\vartheta}$ and $\omega(0) = 0$ then for each positive $\varepsilon$ not exceeding $\min \left\{ \phi^{-1}\left(\frac{\psi(r) - \omega(r)}{2(\vartheta + 1)}\right), \frac{\psi(r)}{\vartheta} \right\}$ the following holds:
$$
\mathcal{O}_{\varepsilon}(\Omega) \in \mathcal{W}^{\omega+(\vartheta + 1)\phi(\varepsilon)}_{r_2, \vartheta},
$$
where $ r_2 = \psi^{-1}\left( \frac{\psi(r)}{2}\right)$.
\end{corollary}

\begin{proof}
Fix $y \in M$ and set $Y = \chi_y \left(\Omega \cap \mathcal{B}_{3\psi(r)}(y)\right)$. Since $Y$ satisfies uniform $\omega$-cusp condition with parameter $r$ at the point $\chi_y(y)$, then according to lemma \ref{lem_39} for each positive $\varepsilon \leq \psi(r) \vartheta$ the set $B_{\varepsilon/\vartheta}$ also satisfies uniform $\omega$-cusp condition with parameter $r_2 = \psi^{-1}(\psi(r) / 2)$ at the point $\chi_y(y)$. Setting $Z = Y^{\varepsilon / \vartheta}$, $R = r_2$, $\rho = \psi(R) = \psi(r)/2$ one obtains from lemma \ref{lem_tran} that
$$
x \in B_{2\rho} (y) \backslash Z \,\,\Rightarrow \,\,x + \phi\left(\left(\vartheta - \frac{1}{\vartheta}\right)\varepsilon\right) \xi_y \notin Z^{\left(\vartheta - \frac{1}{\vartheta}\right) \varepsilon} = \left(Y^{\varepsilon/\vartheta}\right)^{\left(\vartheta - \frac{1}{\vartheta}\right) \varepsilon} = Y^{\frac{\varepsilon}{\vartheta} + \left(\vartheta - \frac{1}{\vartheta}\right) \varepsilon} = Y^{\vartheta \varepsilon}.
$$
now because of lemma \ref{lem_39} for each positive $\varepsilon < \psi(r) /\vartheta$ the set $Y^{\vartheta \varepsilon}$ satisfies uniform $\omega$-cusp condition with parameter $r_2 = \psi^{-1}(\psi(r) / 2)$ at the point $\chi_y(y)$. Thus
\begin{align*}
x \in \chi_y \left[ \mathcal{B}_{\psi(r)}(y) \backslash \mathcal{O}_{\varepsilon}(\Omega) \right] \subset B_{2\rho} (y) \backslash Z &\Rightarrow x + \phi\left((\vartheta - 1/\vartheta)\varepsilon\right) \xi_y \notin Y^{\vartheta \varepsilon} \Rightarrow \\
x + \phi\left((\vartheta - 1/\vartheta)\varepsilon\right) \xi_y + \mathcal{C}_{\omega, r}(\xi_y) \subset \mathbb{R}^d \backslash Y^{\vartheta \varepsilon} &\Rightarrow x + (2\vartheta + 1)\phi\left(\varepsilon\right) \xi_y + \mathcal{C}_{\omega, r}(\xi_y) \subset \mathbb{R}^d \backslash Y^{\vartheta \varepsilon} \Rightarrow\\
x +  \mathcal{C}_{\omega + (2\vartheta + 1)\phi\left(\varepsilon\right), r/2}(\xi_y) \subset \mathbb{R}^d \backslash Y^{\vartheta \varepsilon} \Rightarrow x +  &\mathcal{C}_{\omega + (2\vartheta + 1)\phi\left(\varepsilon\right), r/2}(\xi_y) \subset \chi_y \left( \mathcal{B}_{3\psi(r)}(y) \backslash \mathcal{O}_{\varepsilon}(\Omega)\right),
\end{align*}
if  $\varepsilon \leq \phi^{-1}\left(\frac{\psi(r) - \omega(r)}{2(\vartheta + 1)}\right)$.
\end{proof}

For negative $\eta$ introduce $\mathcal{O}_{\eta}(\Omega) \stackrel{\mathrm{def}}{=} \{y \in \Omega\, |\, \mathcal{O}_{|\eta|}(y) \cap \partial \Omega = \varnothing \}$.

\begin{corollary}
\label{cor_tran}
If $\omega(0) = 0$ and $0 < -\eta, \varepsilon < \phi^{-1}\left( \frac{\psi(R)}{4(\vartheta + 1)} \right)$ then
$$
x \in \chi_y\left( \mathcal{B}_{2\psi(r)}(y) \backslash \mathcal{O}_{\eta}(\Omega)\right) \Rightarrow  x + ( \vartheta + 1 ) (\phi(\varepsilon) + \phi(- \eta))  \xi_y \notin \chi_y\left( \mathcal{O}_{\varepsilon}(\Omega) \cap \mathcal{B}_{3\psi(r)}(y)\right).
$$
\end{corollary}

Similarly to checking of the previous corollary by setting $Z = \chi_y \left(\Omega \cap \mathcal{B}_{3\psi(r)}(y)\right)$, $R = r$, $\rho = \psi(r)$ and applying lemma \ref{lem_tran} one has
$$
x \in \chi_y \left(\mathcal{O}_{\eta}(\Omega) \cap \mathcal{B}_{2\psi(r)}(y)\right) \subset B_{2\psi(r)} (y) \backslash Y^{\vartheta \eta} \,\,\Rightarrow \,\,x + \left[ \phi(-\vartheta \eta) + \phi(\vartheta \varepsilon)  \right] \xi_y \notin Y^{\vartheta \varepsilon},
$$
it remains to note that $Y^{\vartheta \varepsilon} \subset \chi_y\left( \mathcal{O}_{\varepsilon}(\Omega) \cap \mathcal{B}_{3\psi(r)}(y)\right)$.

\section{Proof of the basic results.}

\subsection{Resolvent convergence.}

\begin{lemma}
\label{lem_41}
If $\Omega \in \mathcal{W}^{\omega}_{r,\vartheta}$, $\omega(0) = 0$, $0 < -\eta, \varepsilon < \phi^{-1}\left( \frac{\psi(R)}{4(\vartheta + 1)} \right)$, $\varepsilon \leq 
\mathrm{dist}(\Omega, \partial M_O)$. Then for arbitrary $f \in L_2(M)$ and $\Lambda = \mathcal{O}_{\varepsilon}(\Omega) \backslash \mathcal{O}_{\eta}(\Omega)$ there exists $w_\eta \in V_{\mathcal{O}_{\eta}(\Omega)}$ and $\bar{C} = \bar{C}(M_O, \vartheta, r, \mathcal{A})$ such that
$$
\|u_\varepsilon - w_\eta \|_V^2 \leq \bar{C}\cdot \left( \phi (\varepsilon) + \phi (-\eta) \right)  \|u_\varepsilon\|_{V_{\mathcal{O}_{3 \vartheta \psi(r)}(\Lambda)}} \left( \|u_\varepsilon\|_{V_{\mathcal{O}_{3 \vartheta \psi(r)}(\Lambda)}} + \| f \|_{L_{\mathcal{O}_{3 \vartheta \psi(r)}(\Lambda)}} \right),
$$
where $u_\varepsilon = \mathcal{G}(f; V_{\mathcal{O}_{\varepsilon}(\Omega)})$.
\end{lemma}

\begin{proof}
Since $\Omega \in \mathcal{W}^{\omega}_{r,\vartheta}$ it follows from corollary\ref{obvious} that $\mathcal{O}_{\varepsilon}(\Omega) \in \mathcal{W}^{\omega + (\vartheta + 1)\phi(\varepsilon)}_{r_2,\vartheta}$, $r_2 = \psi^{-1}\left( \frac{\psi(r)}{2} \right)$. Applying theorem \ref{th_38} for $Z = \mathcal{O}_{\varepsilon}(\Omega)$ one can find $\tilde{C} = \tilde{C}(M_O, r, \vartheta, \mathcal{A})$, such that the following estimate holds
\begin{align*}
\|v - v^y_{\nu(y)}\|^2_{V_{\mathcal{B}_{\rho}(y)}} \leq \tilde{C} \|v\|_{V_{\mathcal{B}_{3 \rho}(y)}} \left[  \|v\|_{V_{\mathcal{B}_{3 \rho}(y)}} + \| f \|_{L_{Z \cap \mathcal{B}_{2\rho}(y)}} \right] \cdot (\vartheta + 1)(\phi(\varepsilon) + \phi(-\eta)) = \\
\tilde{C} \cdot (\vartheta + 1)(\phi(\varepsilon) + \phi(-\eta)) \left( \int_Z \textbf{A}(\nabla v, \nabla v) d\mu  + \left(\int_{\mathcal{B}_{3 \rho}(y)} \!\!\!\!\!\!\!\textbf{A}(\nabla v, \nabla v) d\mu \int_{Z\cap \mathcal{B}_{2\rho}(y)} f d\mu \right)^{1/2}\right) \leq\\
\sqrt{2}\tilde{C} \cdot (\vartheta + 1)(\phi(\varepsilon) + \phi(-\eta)) \left[(1+\kappa) \int_Z \textbf{A}(\nabla v, \nabla v) d\mu + \frac{1}{\kappa} \int_{\mathcal{B}_{3 \rho}(y)} f d\mu \right],
\end{align*}
where $v = u_\varepsilon$, $\nu(y) = (\vartheta + 1)(\phi(\varepsilon) + \phi(-\eta)) \xi_y$. Thus, it follows from \ref{cor_tran} that one can use proposition \ref{prop_36} for
\begin{align*}
Y = \mathcal{O}_{\varepsilon}(\Omega),\,\,\,\,\, X = \mathcal{O}_{\eta}(\Omega),\,\,\,\,\, v=u_\varepsilon, \,\, \,\,\, \nu(y) = (\vartheta + 1)(\phi(\varepsilon) + \phi(-\eta)) \xi_y\\
H = \sqrt{2}\tilde{C} \cdot (\vartheta + 1)(\phi(\varepsilon) + \phi(-\eta)) \left[(1+\kappa) \textbf{A}(\nabla v, \nabla v)  + \frac{1}{\kappa} f \right],
\end{align*}
and get that there exists a function $w_\eta \in V_{\mathcal{O}_{\eta}(\Omega)}$ such that
$$
\|u_\varepsilon - w_\eta \|_V^2 \leq  \sqrt{2}\tilde{C} \cdot (\vartheta + 1)(\phi(\varepsilon) + \phi(-\eta)) \cdot \left( (1 + \kappa)\|u_\varepsilon\|^2_{V_{\Lambda^{3\psi(r)}}} + \kappa^{-1} \| f \|_{L_{\Lambda^{3\psi(r)}}}^2 \right).
$$
Choosing $\kappa = \frac{\| f \|_{L_{\Lambda^{3\psi(r)}}}}{\|u_\varepsilon\|_{V_{\Lambda^{3\psi(r)}}}}$ we obtain the required inequality.
\end{proof}

Now the following resolvent continuity will be shown with respect to domain perturbation.

\begin{theorem}
\label{th_42}
Let $\Omega_1 \in \mathcal{W}^{\omega}_{r,\vartheta}$, $\omega(0) = 0$, $f\in L_2(M)$, $u_i = \mathcal{G}(f;V_{\Omega_i})$, $i = 1,2$, then for sufficiently small $\epsilon = e(\Omega_2 \Delta \Omega_1, \partial \Omega_1)$ there exists a constant $\Gamma = \Gamma(M_O, \mathcal{W}^{\omega}_{r,\vartheta}, \mathcal{A})$ such that
\begin{equation}
\label{eq_20}
\| u_1 - u_2 \|_V^2 \leq \Gamma \cdot \phi(\epsilon) \| f \|_{L'} \| f \|_{V'}.
\end{equation}
If in addition $\Omega_2 \in \mathcal{W}^{\omega}_{r,\vartheta}$ and $
d_{\mathcal{HS}}(\Omega_1, \Omega_2)$ is sufficiently small then
$$
\| u_1 - u_2 \|_V^2 \leq \Gamma \cdot \phi(d_{\mathcal{HS}}(\Omega_1, \Omega_2)) \| f \|_{L'} \| f \|_{V'}.
$$
\end{theorem}

\begin{proof}
Set
$$
\varepsilon = e(\Omega_2, \Omega_1), \,\,\,\, \check{\varepsilon} = \check{e}(\Omega_2, \Omega_1), \,\,\,\, \eta = e(\Omega_1, \Omega_2), \,\,\,\, \check{\eta} = \check{e}(\Omega_1, \Omega_2).
$$
To establish (\ref{eq_20}) we use inequality (\ref{eq_9}) of lemma \ref{lem_32}, where
$$
V_1 = V_{\Omega_1}, \,\,\,\, V_2 = V_{\Omega_2}, \,\,\,\, V^{1,2} = V_{\mathcal{O}_{\varepsilon}(\Omega_1)}, \,\,\,\, \alpha = \beta = 1.
$$
As $V_{\mathcal{O}_{-\check{\eta}}(\Omega_1)} \subset V_1 \cap V_2$ and $d(u^{1,2}, V_2) \leq d(u^{1,2}, V_{\mathcal{O}_{-\check{\eta}}(\Omega_1)})$ one has that
$$
\| u_1 - u_2 \|_{V} \leq \left( d_V(u^{1,2}, V_{\Omega_1}) + d_V(u^{1,2}, V_{\Omega_2}) \right) \leq 2 d_V\left(u^{1,2}, V_{\mathcal{O}_{-\check{\eta}}(\Omega_1)}\right).
$$
Now to get estimate (\ref{eq_20}) it is sufficient to use lemma \ref{lem_41} and inequality (\ref{Lax}).

If in addition $\Omega_2 \in \mathcal{W}^{\omega}_{r, \vartheta}$ then interchanging $\Omega_1$ and $\Omega_2$ in the above arguments one obtains the estimate
$$
\| u_1 - u_2 \|_V^2 \leq \Gamma \cdot \phi\left(\min\{ e(\Omega_1 \Delta \Omega_2, \partial \Omega_1),  e(\Omega_1 \Delta \Omega_2, \partial \Omega_2)\}\right) \| f \|_{L'} \| f \|_{V'}.
$$
We notice incidentally that \ref{lem_312} enables one to obtain sufficient smallness of $d^{\mathcal{HP}}(\Omega_1, \Omega_2)$. Hence it sufficient to establish the following estimates
\begin{align}
\label{haus1}
\| u_1 - u_2 \|_V^2 \leq \Gamma \cdot \phi\left( d^\mathcal{H}(\Omega_1, \Omega_2) \right) \| f \|_{L'} \| f \|_{V'},\\
\label{haus2}
\| u_1 - u_2 \|_V^2 \leq \Gamma \cdot \phi\left( d_\mathcal{H}(\Omega_1, \Omega_2) \right) \| f \|_{L'} \| f \|_{V'}.
\end{align}
Now, the first one follows from lemma \ref{lem_41}, inequality (\ref{Lax}) and estimate (\ref{eq_8}) in lemma \ref{lem_32}:
\begin{align*}
V_{\mathcal{O}_{-\check{\eta}} (\Omega_1)} \subset V_1 \cap V_2, \,\,\,\, d_V(u_1, V_1 \cap V_2)^2 \leq d_V(u_1, V_{\mathcal{O}_{-\check{\eta}}(\Omega_1)})^2 \leq \Gamma \cdot \phi\left( \check{\eta} \right) \| f \|_{L'} \| f \|_{V'},\\
V_{\mathcal{O}_{-\check{\varepsilon}} (\Omega_2)} \subset V_1 \cap V_2, \,\,\,\, d_V(u_2, V_1 \cap V_2)^2 \leq d_V(u_2, V_{\mathcal{O}_{-\check{\varepsilon}}(\Omega_2)})^2 \leq \Gamma \cdot \phi\left( \check{\varepsilon} \right) \| f \|_{L'} \| f \|_{V'}.
\end{align*}
As regards (\ref{haus2}) it follows again from lemma \ref{lem_41}, inequality (\ref{Lax}) and corollary \ref{cor_33} where
$$
V^{1,2} = V_{\mathcal{O}_{\varepsilon}(\Omega_1)}, \,\,\,\, V^{2,1} = V_{\mathcal{O}_{\eta}(\Omega_2)}.
$$
\end{proof}

Resolvent continuity in $\Omega$ established above enables one to estimate distance between eigenspaces of the spectral boundary value problems for $\mathcal{A}$ with close domains $\Omega_1$ and $\Omega_2$. Namely, let $\partial \Omega_1 \in C^{0,\omega}$, $\nu_k$ be $k$-th (in decreasing order) eigenvalue of the operator $\mathcal{G}(\cdot, V_{\Omega_1})$ without taking multiplicity into account,  $E_k(\Omega_1)$ be the corresponding eigenspace, number $r > 0$ be such that $B_{2r}(\nu_k) \cap \mathrm{spec}(\mathcal{G}(\cdot, V_{\Omega_1})) = \{ \nu_k \}$. Then there exists $\delta = \delta(r, \nu_k) > 0$ so that as soon as $\epsilon = e(\Omega_1 \Delta \Omega_2, \partial \Omega_1) < \delta$ for the generalized angle (see, for instance, \cite{Kato})
$$
\tilde{\delta}_V(A, B) = \max \left\{ \sup_{u \in A, \| u \|_V = 1} d_V (u, B), \sup_{u\in B, \|u\|_V = 1} d_V(u, A) \right\}
$$
between subspaces $A, B \subset V$ the following estimate takes place
$$
\tilde{\delta}_V(E_k(\Omega_1), E_k(\Omega_2)) \leq \Gamma \cdot C(\nu_k, r) \cdot \phi(\epsilon)^{1/2},
$$
where $E_k(\Omega_2)$ is the range of the operator
$$
\frac{1}{2\pi i} \int_{|\nu_k - \xi| = r} \left( \mathcal{G}(\cdot, V_{\Omega_2}) - \xi \right)^{-1} d\xi.
$$

\subsection{Proof of the Theorem \ref{result}}

By means of Proposition \ref{classif} for domain $\Omega_1$ with the boundary of class $C^{0,\omega}$ one has a class $\mathcal{W}^{C \omega}_{r,\vartheta}$ (without loss of generality one can assume that $C = 1$). Let $u_n^{(i)} \in V_{\Omega_1}$, $i = 1,2$, be eigenfunctions associated with eigenvalues $\lambda_n^{(i)}$ of the problem (\ref{SpecDir}) considered in domains $\Omega_i$ and $\bar{u}_n^{(1)} \in \Sob(\Omega_2)$ be weak solution of the equation
$$
\mathcal{A} u = \lambda_{n}^{(1)} u_n^{(1)},\,\,\,\,\,\, u\in \Sob(\Omega_2).
$$
Taking advantage of the Theorem \ref{th_42} one has the estimate
$$
\| u_n^{(1)} - P_{\Omega_2} u_n^{(1)} \|^2_{V} \leq  \| u_n^{(1)} -  \bar{u}_n^{(1)} \|^2_{V} \leq \Gamma \cdot \phi (e(\Omega_1 \Delta \Omega_2, \partial \Omega_1)) \| \lambda_{n}^{(1)} u_n^{(1)} \|_{L'} \| \lambda_{n}^{(1)} u_n^{(1)} \|_{V'},
$$
besides the following inequality holds:
$$
\| \lambda_{n}^{(i)} u_{n}^{(i)} \|_{H^{-1}(M_O)} \leq C(n,\mathcal{A}, \mathrm{p}) \sqrt{\lambda_{n}^{(i)}} \| u_{n}^{(i)} \|_{L_2(M_O)}.
$$
Since $\lambda_{n}^{(1)} \leq \gamma_n$ does not exceed the $n$-th eigenvalue of the problem (\ref{SpecDir}), with $\Omega$ being a ball contained in $\Omega_1 \cap \Omega_2$, one has
$$
\| u_j^{(1)} - P_{\Omega_2} u_j^{(1)} \|^2_{V} \leq  C_j \phi(e(\Omega_1 \Delta \Omega_2, \partial \Omega_1)), \,\,\,\,\,\, j = \overline{1, \ldots, n}
$$
Now applying lemma \ref{Birkhoff} one derives the estimate from below for $\lambda_n^{(1)}$. To obtain similar estimate for $\lambda_n^{(2)}$ we write
$$
\| u_n^{(2)} - P_{\Omega_1} u_n^{(2)} \|^2_{V} \leq  \| u_n^{(2)} -  \bar{u}_n^{(2)} \|^2_{V} \leq \Gamma \cdot \phi (e(\Omega_1 \Delta \Omega_2, \partial \Omega_1)) \| \lambda_{n}^{(2)} u_n^{(2)} \|_{L'} \| \lambda_{n}^{(2)} u_n^{(2)} \|_{V'},
$$
where $\bar{u}_n^{(2)} = \mathcal{G}( \lambda_{n}^{(2)} u_n^{(2)}; V_{\Omega_1})$. Carrying on the above arguments we obtain the required conclusion

If in addition $\Omega_2 \in \mathcal{W}^{\omega}_{r, \vartheta}$ then
$$
\| u_n^{(1)} - P_{\Omega_2} u_n^{(1)} \|^2_{V} \leq \Gamma \cdot \phi (d_{\mathcal{HS}}(\Omega_1, \Omega_2)) \| \lambda_{n}^{(1)} u_n^{(1)} \|_{L'} \| \lambda_{n}^{(1)} u_n^{(1)} \|_{V'}
$$

\begin{corollary}[Manifold version of Burenkov--Lamberti theorem]
\label{cor_bur}
Let operator $\mathcal{A}$ on the manifold $(M,g)$ satisfies conditions (\textbf{A1}) -- (\textbf{A2}), $\Omega_1, \Omega_2 \in \mathcal{W}^{\omega}_{r,\vartheta}$, $\omega(0)=0$. Then there exists constants $C_n = C_n(M, \omega, r, \vartheta, \mathcal{A} )>0$,  $\delta_0  = \delta_0(M_O, \omega, r, \vartheta )>0$ such that conditions $\Omega_1 \Subset M_O$, $d_{\mathcal{HS}}(\Omega_1, \Omega_2) \leq \delta_0$ imply inequality
$$
|\lambda_n^{(1)} - \lambda_n^{(2)}| \leq C_n ( \omega (d_{\mathcal{HS}}(\Omega_1, \Omega_2)) + d_{\mathcal{HS}}(\Omega_1, \Omega_2)),
$$
where $\{\lambda^{(i)}_n\}$ are eigenvalues of the problem (\ref{SpecDir}) for domains $\Omega_i$ indexed in ascending order with multiplicities taken into account.
\end{corollary}


\begin{thebibliography}{99}


\bibitem{Ambrosio}
\textit{Ambrosio L., Fusco N., Pallara D.} Functions of Bounded Variation and Free Discontinuity Problems. Oxford: Oxford Math. Monographs, Clarendon Press, 2000.

\bibitem{Barbatis}
\textit{Barbatis G., Lamberti P.D.} Spectral stability estimates
for elliptic operators subject to domain transformations with non-uniformly bounded
gradients. Mathematika. 2012. 58, no. 2, 324--348.

\bibitem{Birkhoff}
\textit{Birkhoff~G., de Boor~C., Swartz~B., and Wendroff~B.} Rayleigh-Ritz approximation by piecewise cubic polynomials. // SIAM J. Numer. Anal. 1966. \textbf{3}. 188–-203.

\bibitem{Brezis}
\textit{Brezis~H.} Analyse fonctionnelle – Th\'{e}orie et applications, P.: Masson, 1983.

\bibitem{Bucur}
\textit{Bucur~D., Buttazzo~G.} Variational Methods in Shape Optimization Problems. Basel, Boston: Progress in Nonlinear Differential Equations and Their Applications, \textbf{65}, Birkh\"auser. 2005.

\bibitem{Burenkov}
\textit{Burenkov~V.I., Davies~E.B.} Spectral stability of the Neumann Laplacian. // J. Differential Equations. 2002. \textbf{186}(2). 485–-508.

\bibitem{BurenkovLamberti}
\textit{Burenkov~V.I., Lamberti~P.D., and Lanza de Cristoforis~M.} Spectral stability of nonnegative selfadjoint operators. // Sovrem. Mat. Fundam. Napravl. 2006. \textbf{15}. 76--111; translation in J. Math. Sci. (N. Y.) 2008. \textbf{149}(4) 1417–-1452.

\bibitem{BurLam}
\textit{Burenkov, V.I,, Lamberti, P.D.} Spectral stability of higher order
uniformly elliptic operators. Sobolev spaces in mathematics. II, 69-102, Int.
Math. Ser. (N. Y.), 9, Springer, New York, 2009.

\bibitem{ChenaisEng}
\textit{Chenais~D.} On the Existence of a Solution in a Domain Identification Problem // J. Math. Anal. Appl. 1975. \textbf{52}. 189--219.

\bibitem{EvansGar}
\textit{Evans L.C., Gariepy R.F.} Measure Theory and Fine Properties of Functions. CRC Press. 1991.

\bibitem{Frehse}
\textit{Frehse~J.} Capacity methods in the theory of partial differential equations. // Jahresber. Deutsch. Math.-Verein. 1982. \textbf{84}(1), 1--44.

\bibitem{HenrotEng}
\textit{Henrot~A.} Extremum Problems for Eigenvalues of Elliptic Operators // Basel-Boston-Berlin: Birkh\"auser Verlag. 2006.

\bibitem{Kato}
\textit{Kato~T.} Perturbation Theory for Linear Operators. N.Y.: Springer Verlag, 1966.

\bibitem{Burago}
\textit{Kelly~P.J., Weiss~M.L.} Geometry and Convexity: A Study in Mathematical Methods. Wiley. 1979.

\bibitem{Lemenant}
\textit{Lemenant~A., Milakis~E., Spinolo~L.} Spectral stability estimates for the Dirichlet and Neumann Laplacian in rough domains. // Journal of Functional Analysis. 2013. \textbf{264} (9).  2097--2135.

\bibitem{Maslov}
\textit{Маслов~В.П.} Теория возмущения линейных операторных уравнений и проблема малого параметра в дифференциальных уравнениях // ДАН 1956 \textbf{111} 531--534.

\bibitem{Mosco}
\textit{Mosco~U.} Approximation of the solutions of some variational inequalities.
// Ann. Scuola Normale Sup. (Pisa). 1967. \textbf{21}. 373–-394.

\bibitem{Savare}
\textit{Savar\'{e}~G., Schimperna~G.} Domain perturbations estimates for the solutions of second order elliptic equations // J. Math. Pures Appl. 2002. \textbf{81}(11). 1071--1112.

\bibitem{StepinTsylin}
\textit{Stepin~A.M., Tsylin~I.V.} On boundary value problems for elliptic operators in the case of domains on manifolds.
// Doklady Mathematics. 2015. \textbf{92} (1). 428--432.

\bibitem{Tsylin}
\textit{Tsylin I.V.} Continuity of eigenvalues of the Laplace operator according to domain // Moscow University Mathematics Bulletin. 2015. \textbf{3}. 136--140.

\end{thebibliography}
\end{document}